\documentclass{amsart}

\usepackage{amssymb, array} 
\usepackage{mathtools} 
\usepackage{tikz}
\usetikzlibrary{calc}

\newcommand{\ve}{\varepsilon}

\newtheorem{theorem}{Theorem}[section]

\newtheorem{lemma}[theorem]{Lemma}

\title[The perfect conductivity problem]
{The perfect conductivity problem \\
with arbitrary vanishing orders  \\
and non-trivial topology}

\author[M. Sherman]{Morgan Sherman}
\address{Department of Mathematics, California Polytechnic State University, San Luis Obispo, CA
93407}
\author[B. Weinkove]{Ben Weinkove}
\address{Department of Mathematics, Northwestern University, 2033 Sheridan Road, Evanston, IL 60208}

\thanks{Research supported in part by NSF grant DMS-2005311.}

\begin{document}

\maketitle

\vspace{-20pt}

\begin{abstract}
The perfect conductivity problem concerns optimal bounds for the magnitude of an electric field in the presence of almost touching perfect conductors.  This reduces to obtaining gradient estimates for harmonic functions with Dirichlet boundary conditions in the narrow region between the conductors.  In this paper we extend estimates of Bao-Li-Yin to deal with the case when the boundaries of the conductors are given by graphs with arbitrary vanishing orders.  Our estimates allow us to deal with globally defined narrow regions with possibly non-trivial topology.   

We also prove the sharpness of our estimates in terms of the distance between the perfect conductors.  The precise optimality statement we give is new even in the setting of Bao-Li-Yin.
\end{abstract}

\section{Introduction}

Let $\Omega \subset \mathbb{R}^n$ be a bounded open domain with
smooth boundary.
Let $D_1, D_2$ be disjoint domains in $\Omega$
with smooth boundaries 
which are a distance $\varepsilon > 0$ apart,
and a distance at least
$d \gg \varepsilon$ from $\partial \Omega$.
We write $\tilde\Omega = \Omega \setminus (\overline{D}_1 \cup \overline{D}_2)$.

Fix a smooth function $\varphi$ on $\partial \Omega$.
The setup of the \emph{perfect conductivity problem} is the following PDE:
\begin{equation}
  \label{perf cond prob}
  \begin{aligned}
    & \Delta u = 0 \quad \text{ in } \tilde\Omega \\
    & u_+ = u_- \quad \text{ on } \partial D_1 \cup \partial D_2 \\
    & \nabla u = 0 \quad \text{ on } D_1 \cup D_2 \\
    & \int_{\partial D_i}
      \left.\frac{\partial u}{\partial \nu}\right|_+
      = 0 \quad i = 1, 2 \\
    & u = \varphi \quad \text{ on } \partial\Omega . 
  \end{aligned}
\end{equation}
Here $u_+$ and $u_-$ refer to the limits of $u$ from outside and
inside (respectively) the sets $D_1, D_2$.
The third equation implies that $u = C_1$ on $D_1$
and $u = C_2$ on $D_2$, for constants $C_1, C_2$.
The function
$\frac{ \partial u }{ \partial \nu } |_+$
in the fourth line of (\ref{perf cond prob})
is the derivative of $u$ in the direction $\nu$,
the unit outward normal vector on $\partial D_i$.  Namely, at $x_0 \in \partial D_i$ it is 
the limiting value of
$ \nabla u(x) \cdot \nu (x_0)$ as $x \to x_0$
through values within $\tilde\Omega$.

The question is: what happens to
$|\nabla u|$ as $\varepsilon \to 0$? 

This problem has a physical interpretation in terms of electrical conductivity.  The domains $D_1$ and $D_2$ represent perfect conductors and $u$ represents the electric potential.  The question is then how the  magnitude  of the electric field $\nabla u$ may blow up as the perfect conductors approach each other.  When $n=2$, there is also a physical interpretation in terms of composite materials in which $\Omega$ represents the cross-section of a fiber-reinforced composite (here $D_1$ and $D_2$ represent the embedded fibers).   In this case,  the electric potential and field are replaced by the out-of-plane elastic displacement and the stress tensor respectively.  For background and more details, we refer the reader to \cite{BASL, BV, LN, LV} and the references therein.

The standard setting in the literature is that $D_1$ and $D_2$ are strictly convex sets a distance $\varepsilon>0$ apart (see figure \ref{balls}).  In this case there is a single ``narrow region'' of $\tilde{\Omega}$ between the two sets $D_1$, $D_2$.  After translation, this region is given by points $x=(x', x_n)$ with $f(x') \le x_n \le g(x')$ for $|x_1|, \ldots, |x_{n-1}| \le r$, for a uniform $r>0$, where  $f$ and $g$ are smooth functions with $(g-f)(0')=\ve$,  $g-f \ge \ve$ and $D^2 g > 0 > D^2f$.  Here we use the usual notation $x' =(x_1, \ldots, x_{n-1})$.  Points outside this narrow region can be characterized by the fact that they are contained in  a ball of uniform radius  lying completely inside $\tilde{\Omega}$.


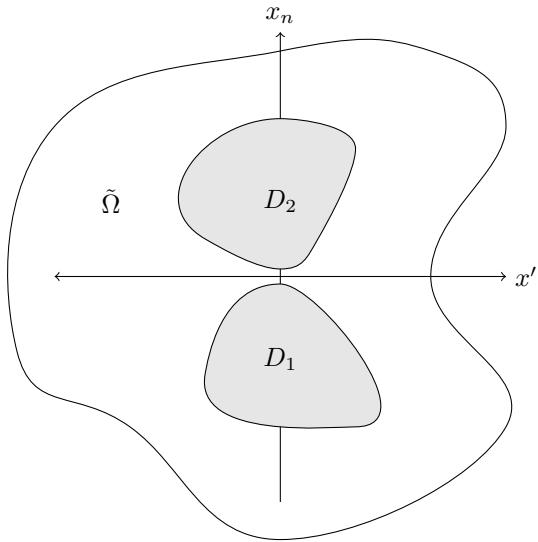
\begin{figure}[ht]
\label{balls}
\begin{tikzpicture}

\path
  (2,0) coordinate (a)
  (3,2) coordinate (b)
  (2,3) coordinate (c)
  (0,3) coordinate (d)
  (-3,2) coordinate (e)
  (-3.5,-1) coordinate (f)
  (-2,-2) coordinate (g)
  (0,-3.5) coordinate (h)
  (3,-2) coordinate(i)
;

\draw plot [smooth cycle, tension=0.75]
  coordinates { (a) (b) (c) (d) (e) (f) (g) (h) (i) };

\node at (-2.25,1) {$\tilde\Omega$};

\draw [<->] (-3,0) -- (3,0) node [right] {$x'$};
\draw [->] (0,-3) -- (0,3.25) node [above] {$x_n$};

\path 
  (0,0.1) coordinate (d2a) 
  (0.5,0.5) coordinate (d2b)
  (1,1.7) coordinate (d2c) 
  (0,2.1) coordinate (d2d) 
  (-1,0.5) coordinate (d2e)
;
\filldraw [fill=gray!20]
  (d2a)
  .. controls +(east:0.3) and +(240:0.3) ..
  (d2b)
  .. controls +(60:0.3) and +(south:0.3) ..
  (d2c)
  .. controls +(north:0.3) and +(east:0.3) ..
  (d2d)
  .. controls +(west:1) and +(150:1) ..
  (d2e)
  .. controls +(-30:0.3) and +(west:0.3) ..
  (d2a)
;
\node at (0,1) {$D_2$};

\path 
  (0,-0.1) coordinate (d1a) 
  (-1,-1.3) coordinate (d1b)
  (1,-2) coordinate (d1c) 
;
\filldraw [fill=gray!20]
  (d1a)
  .. controls +(west:0.8) and +(80:0.2) ..
  (d1b)
  .. controls +(-100:0.9) and +(west:0.3) ..
  (d1c)
  .. controls +(east:1) and +(east:0.5) ..
  (d1a)
;
\node at (0,-1.1) {$D_1$};


\end{tikzpicture}
\caption{The case when $D_1$ and $D_2$
are convex sets a distance $\varepsilon$ apart.}
\label{figure1}
\end{figure}


It has been known for some time that in general the gradient $|\nabla u|$ may blow up as $\ve \rightarrow 0$ \cite{BC, M}.  It was shown in \cite{BC} that $\sup|\nabla u|$ blows up at a rate $\ve^{-1/2}$  for a special solution in $\mathbb{R}^2$.  More general solutions in the case when $D_1$, $D_2$ are disks with  comparable radii in $\mathbb{R}^2$ were dealt with by Ammari-Kang-Lim \cite{AKL} (who gave the lower bound of $\sup|\nabla u|$) and Ammari-Kang-Lee-Lee-Lim \cite{AKLLL} (the upper bound).  Yun \cite{Yun07} extended  \cite{AKL} to more general convex subdomains in $\mathbb{R}^2$ which are positively curved at the closest point.

Bao-Li-Yin \cite{BLY0} then proved a much more general result which allows any dimension $n \ge 2$ and the case when the domains do not necessarily have positively curved boundaries. We now describe their main result.  If there is a single narrow region as above with
\begin{equation} \label{BLYcondition}
\frac{1}{C} (\ve + |x'|^{2\alpha})  \le g(x')-f(x') \le C (\ve + |x'|^{2\alpha}),
\end{equation}
for a uniform $C$ and a constant $\alpha \ge 1$ then $|\nabla u|$ is bounded on $\tilde{\Omega}$ as follows:
\setlength{\extrarowheight}{12pt}
\begin{equation} \label{BLY}
\sup_{\tilde{\Omega}} |\nabla u| \le  \left\{ \begin{array}{ll}\ \displaystyle{\frac{C}{\ve^{\frac{n-1}{2\alpha}}}}, &\quad \textrm{if } n-1<2\alpha \\
 \displaystyle{ \frac{C}{\ve |\log\ve|}}, & \quad \textrm{if } n-1=2\alpha \\
\displaystyle{  \frac{C}{\ve}}, & \quad \textrm{if } n-1>2\alpha. \end{array} \right.
\end{equation}
Bao-Li-Yin \cite{BLY0} also proved the optimality of their estimates under some symmetry assumptions on the domains.  They made use of a linear functional $Q_{\ve}(\varphi)$ which we will describe later  (see Section \ref{sec: opt bds}).

Since then, there have been many further results, refining and extending the estimates (\ref{BLY}) and giving detailed asymptotics, see  \cite{AKLLZ, BLY, BT,  CLX, DL, DZ, HZ, KLY, LWX, LY, Yun09}, for example (this is far from a complete list).

In this paper we give a broad extension of the Bao-Li-Yin estimates in a different direction, allowing for more complicated geometry and topology of the sets $D_1$ and $D_2$.
We will allow the vanishing orders of the boundaries of $D_1$ and $D_2$ to be different in each of the $n-1$ directions $x_1, \ldots, x_{n-1}$.  Namely, we replace the  quantity   $|x'|^{2\alpha}$ in (\ref{BLYcondition}) by  
$$\sum_{j=1}^{n-1} x_j^{2\alpha_j}, \quad \textrm{for constants }\alpha_1, \ldots, \alpha_{n-1} \ge 1.$$
We refer to the constants $\alpha_1, \ldots, \alpha_{n-1}$ as the \emph{vanishing orders} of the boundary, and a key point of this paper is that the $\alpha_j$ need not all be equal.  Crucially, we also allow any number of the $\alpha_j$ to take the value $+\infty$, which we take to mean that $x_j^{2\alpha_j}$ doesn't appear in the sum (we may assume that that $|x_j|<1$).

We also deal with the case when the narrow region is defined by a finite union of sets, each of which is given by the set of points between graphs $f$ and $g$.  This allows the possibility that $D_1$ and $D_2$ are ``close together'' in several regions throughout $\Omega$ (but far away from $\partial\Omega$).   For example there may be a curve (or even higher-dimensional set) of points
on $\partial D_1$ that are at a distance on the order
of $\varepsilon$ from $\partial D_2$.  

Such situations may arise for example when the conductors have non-trivial topology, such as the case of encircled tori, which may be closely touching along a circle of points (see, for example,
figure \ref{fig: encircled tori}).

Our approach will be to cover the narrow region between $D_1, D_2$
with small open sets where we do individually get a simple
picture, and then to piece this together to a global statement.
See figure
\ref{fig: more general setup}.  

\bigskip
{\bf Statement of the main results.} \  The region $\Omega$ is fixed, independent of
$\varepsilon$.  It is convenient to regard $D_1$ and $D_2$ as elements of a smoothly varying family of domains.  We fix a compact set $K$ in $\Omega$, and 
 domains $D^0_1, D^0_2$ with smooth boundaries contained in $K$.  For a small constant $\ve_0>0$ we consider smooth families of domains $\{ D^{\ve}_1\}_{\ve \in [0,\ve_0]}$ and $\{ D_2^{\ve}\}_{\ve\in[0,\ve_0]}$, also contained in $K$ and with smooth boundaries. To make the notion of ``smooth family'' more precise, for $i=1, 2$, write $M^{\ve}_i = \partial D_i^{\ve}$  for $\ve \in [0,\ve_0]$.  These are smooth closed embedded hypersurfaces in $\mathbb{R}^n$.  Then there 
are smooth maps $F_i: M^0_i \times [0,\ve_0] \rightarrow K$ for $i=1, 2$ such that $x \mapsto F_i(x, \ve)$ is a diffeomorphism from $M^0_i$ onto $M^{\ve}_i$ and the identity when $\ve=0$.
 
We assume that $\overline{D_1^{\ve}} \cap \overline{D_2^{\ve}}=\emptyset$ for every $\ve \in (0,\ve_0]$, but the intersection may be nonempty at $\ve=0$.  
We also make an additional assumption:

\medskip
$(*)$ There is a smooth path from $\partial \Omega$ to $\partial D_1^0$ that does not intersect $\overline{D_2^0}$,
and another smooth path from $\partial \Omega$ to $\partial D_2^0$ that does not intersect $\overline{D_1^0}$.
\medskip

This rules out the case when one of the domains $D_1^0, D_2^0$ completely envelops the other.  See the beginning of Section \ref{sectionproof} for more discussion of assumption $(*)$.

Our goal is to obtain optimal bounds, in terms of $\ve$, for the gradient of $u$ solving (\ref{perf cond prob}) for $D_1=D_1^{\ve}$ and $D_2=D_2^{\ve}$.  {\bf For simplicity of notation, in what follows we will drop the superscript $\ve$ and write $D_1$ and $D_2$ instead of $D_1^{\ve}$ and $D_2^{\ve}$.}

We assume there is a constant $c_0>0$
and an open subset $V \subset \tilde\Omega$ such that
$V$ satisfies an interior ball condition of
radius $c_0$.
Specifically we mean by this that for all
$p \in \partial V$ there is a ball $B$ of radius $c_0$
such that $p \in \partial B$ and $B \subset V$.
Thus $V$ consists of points that are ``far away'' from the
narrow region between $D_1$ and $D_2$.
Next we assume that  $\tilde\Omega \setminus V$ can be
covered by open boxes $U_i = U_i^{(r)}$ for $i=1, \ldots, k$,
of a fixed size $r>0$, where for each $i=1, \ldots, k$, 
after possibly translating and rotating the coordinates, 
$U_i = \{ (x_1, \ldots, x_n) \mid \max_j |x_j| < r\}$.
Further, we assume that for each $i=1,\ldots,k$ there are
functions $f_i, g_i$ on the set 
\begin{equation}
  \notag 
  Q_r := \{ x' \mid \max_{j=1, \ldots, n-1} |x_j|<r\}
  \subset \mathbb{R}^{n-1}
\end{equation}
such that 
\begin{equation}
  \notag 
  \tilde U_i \coloneqq U_i \cap \tilde\Omega
  = \{ (x,y) \mid x \in Q_r \text{ and } f_i(x') < x_n < g_i(x') \}.
\end{equation}
See figure \ref{fig: region Ui}.   Moreover, we assume that the boxes $U_i$ overlap sufficiently so that the union of boxes $U_i^{(r/2)}$ of size $r/2$ still covers the region
 $\tilde{\Omega} \setminus V$.


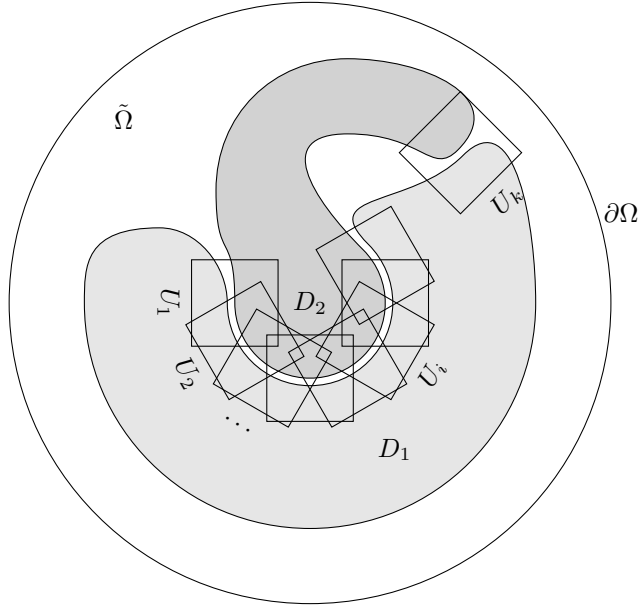
\begin{figure}[h]

\begin{tikzpicture}

\coordinate (A) at (2,2); 
\path
  (A)
  +(135:0.05) coordinate (A2)
  +(-45:0.05) coordinate (A1)
;

\draw circle [radius=4];
\node [anchor=south west, inner sep=1pt] at (15:4) {$\partial\Omega$};
\node at (135:3.5) {$\tilde\Omega$};

\filldraw [fill=gray!20]
  ([shift={(-180:1.1)}] 0,0)
  arc [radius=1.1, start angle=-180, end angle=45]
  .. controls +(135:1) and +(225:1) ..
  (A1)
  .. controls +(45:1) and +(90:1) ..
  (3,0)
  arc [radius=3, start angle=0, delta angle=-180]
  .. controls +(90:0.5) and +(180:1) ..
  (-2,1)
  .. controls +(0:.5) and +(90:0.5) ..
  (180:1.1)
;
\node at (-60:2.25) {$D_1$};

\filldraw [fill=gray!35]
  ([shift={(45:1)}] 0,0)
  arc [radius=1, start angle=45, end angle=-180]
  .. controls +(90:1) and +(-90:1) ..
  (-1.25,1.5)
  .. controls +(90:1) and +(180:1) ..
  (0.5,3.25)
  .. controls +(0:1) and +(45:1) ..
  (A2)
  .. controls +(-135:0.5) and +(0:1) ..
  (0.5,2.25)
  .. controls +(180:1) and +(135:1) ..
  (45:1)
;
\node at (0,0) {$D_2$};

\begin{scope}
[every node/.style={shape=rectangle, minimum size=1.15cm, draw}]

\foreach \t/\n in
{
  -180/1, -150/2, -120/3, -90/4, -60/5, -30/6, 0/7, 30/8
}
  \node [rotate=90+\t] (a\n) at (\t:1) {};
\node [rotate=45] (a0) at (A) {};
\end{scope}

\path
  (a1.south) ++(180:0.3) node [rotate=-90] {$U_1$}
  (a2.south) ++(-150:0.3) node [rotate=-60] {$U_2$}
  (a3.south) ++(-120:0.3) node [rotate=-30] {$\cdots$}
  (a6.south) ++(-30:0.3) node [rotate=60] {$U_i$}
  (a0.south) ++(-45:0.3) node [rotate=45] {$U_k$}
;


\end{tikzpicture}

\caption{A possible configuration of the inclusions $D_1$ and $D_2$.}
\label{fig: more general setup}
\end{figure}


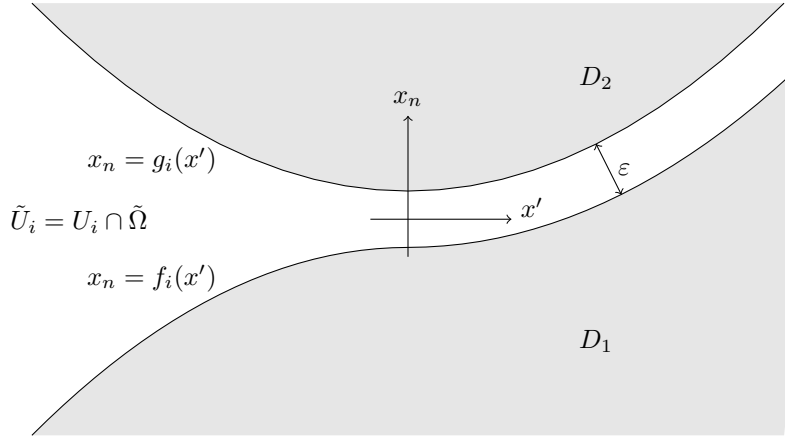
\begin{figure}[h]
\begin{tikzpicture}
[
  region/.style={fill=gray!20},
  boundary/.style={thin}, 
  declare function={
    CC = 1/4; 
    myf(\x) = CC*pow(\x,2);
    myfprime(\x) = 2*CC*\x;
    dist = 0.3;
    norm(\t) = sqrt(1+pow(myfprime(\t),2));
    scmult(\t) = dist/norm(\t);
    mimicfx(\t) = \t + scmult(\t)*myfprime(\t);
    mimicfy(\t) = myf(\t) - scmult(\t);
    lly=-dist-myf(-2);
  },
  scale=2.5
]

\path [use as bounding box]
  (-2,{-dist-myf(-2)}) rectangle (2, {myf(2)});

\fill [region]
  plot [domain=-2:2] ( \x , {myf(\x)} );
\draw [boundary]
  plot [domain=-2:2] ( \x , {myf(\x)} );

\begin{scope}
\clip (-2,-2) rectangle (2,{myf(2)+0.5});

\fill [region]
  plot [domain=-2:0]
       ( \x , {-dist-myf(\x)} )
  --plot [domain=0:2, variable=\t]
         ( { mimicfx(\t) } , { mimicfy(\t) } )
  -- (2,lly)
  -- cycle;
\draw [boundary]
  plot [domain=-2:0]
       ( \x , {-dist-myf(\x)} )
  -- plot [domain=0:2, variable=\t]
         ( { mimicfx(\t) } , { mimicfy(\t) } );
\end{scope}

\node [anchor=center] at (1,{lly+0.5*(-lly-dist)}) {$D_1$};
\node [anchor=center] at (1, {0.6*myf(2)}) {$D_2$};

\node [below left,inner sep=1pt] at (-1,{myf(-1)}) {$x_n=g_i(x')$};
\node [above left,inner sep=1pt] at (-1,{-dist-myf(-1)}) {$x_n=f_i(x')$};
\node at (-1.75,-0.5*dist) {$\tilde U_i = U_i \cap \tilde\Omega$};

\coordinate (A) at (1,{myf(1)});
\coordinate (B) at  ({mimicfx(1)}, {mimicfy(1)});
\draw [<->]
  (A) -- (B)
  node [midway, right] {$\varepsilon$};

\begin{scope} [shift={(0,-0.5*dist)}]
\draw [->] (0,-0.2) -- +(0,0.75) node [above] {$x_n$};
\draw [->] (-0.2,0) -- ++(0.75,0) node [anchor=210] {$x'$};
\end{scope}

\end{tikzpicture}
\caption{The neighborhood $\tilde U_i$ in the narrow region}
\label{fig: region Ui}
\end{figure}


Here the graphs of the functions $f_i$ and $g_i$ will
describe the portions of the boundaries of $D_1$ and $D_2$
which lie within $U_i$.  
We assume that the boxes $U_i$ are chosen centered near points
where $D_1$ and $D_2$ are closest together.
Specifically we assume that
there is a constant $C>0$, independent of $\varepsilon$,
such that for all $i = 1, \ldots, k$ there are
positive numbers (or infinities, see below)
$\alpha^i_{1}, \ldots, \alpha^i_{n-1}$ such that 
\begin{equation}
  \label{fi and gi}
  \frac 1 C
  \bigl( \varepsilon + \sum_{j=1}^{n-1} x_j^{2 \alpha^i_j} \bigr)
  < g_i(x') - f_i(x') <
  C
  \bigl( \varepsilon + \sum_{j=1}^{n-1} x_j^{2 \alpha^i_j} \bigr)
  \quad \text{for all } x' \in Q_r .  
\end{equation}
We interpret the quantity
$x_j^{2 \alpha^i_j}$ as
$(x_j^2)^{\alpha^i_j}$, so that, for example,
$x^{2 \cdot \frac12} = | x |$.  
We allow the possibility $\alpha^i_j = \infty$,
in which case we interpret the above
expression containing $x_j^{2 \alpha^i_j}$ to mean
this term does not appear.
We set $\alpha^i \coloneqq (\alpha^i_j, \ldots, \alpha^i_{n-1})$.  

Here and throughout this article we use
$C$ to denote a positive constant that is independent of
$\varepsilon$,
and it may change from line to line.  

We assume that the boxes $U_1, \ldots, U_k$ are fixed independent of $\ve$ but that 
 $f_i, g_i$ may depend on $\varepsilon$.  Nevertheless, it follows from our assumptions that these functions and their derivatives are uniformly bounded independent of $\ve$, and in particular
\begin{equation}
  \notag 
  | f_i(x') |,
  | g_i(x') |,
  | \nabla f_i(x') |,
  | \nabla g_i(x') |
  \le C
  \quad \text{for all } x' \in Q_r. 
\end{equation}

To state our main theorem we define,
for each $i=1,\ldots, k$, 
\begin{equation}
  \notag 
  \gamma_i := \sum_{j=1}^{n-1} \frac 1 {2 \alpha^i_j},
  \ i = 1, \ldots, k 
\end{equation}
(where, if $\alpha^i_j = \infty$, then $\frac 1 {2\alpha^i_j} = 0$),
and set
\begin{equation}
\label{gamma def}
  \gamma := \min_{i=1, \ldots, k} \gamma_i .  
\end{equation}

The main result is:
\begin{theorem} With assumptions  as above, let $u$ solve (\ref{perf cond prob}) for $D_1 = D_1^{\ve}$ and $D_2=D_2^{\ve}$. Given any smooth boundary data $\varphi$ on $\partial \Omega$ there exists a constant $C$ independent of $\ve>0$ such that 
\label{main thm}
\begin{equation}
  \notag 
 \sup_{\tilde{\Omega}} |\nabla u| \le \left\{ \begin{array}{ll}\ \displaystyle{\frac{C}{\ve^{\gamma}}}, &\quad \textrm{if } \gamma<1 \\
 \displaystyle{ \frac{C}{\ve |\log \ve|}}, & \quad \textrm{if } \gamma=1 \\
\displaystyle{  \frac{C}{\ve}}, & \quad \textrm{if } \gamma>1. \end{array} \right.
\end{equation}
\end{theorem}

One can check that the exponents match with the Bao-Li-Yin estimate (\ref{BLY}) in the case of a single coordinate patch with $\alpha:= \alpha^1_1 = \cdots = \alpha^1_{n-1}$.  

Our next result shows that these estimates are optimal in terms of $\ve$.

\begin{theorem} \label{thmoptimal}
With the assumptions as above,  there exists smooth boundary data $\varphi$ on $\partial \Omega$ such that the following holds.  If $u$ solves (\ref{perf cond prob}) for $D_1 = D_1^{\ve}$ and $D_2=D_2^{\ve}$ then there exists a constant $C>0$ such that for $\ve>0$ sufficiently small,
\begin{equation}
  \notag 
 \sup_{\tilde{\Omega}} |\nabla u| \ge \left\{ \begin{array}{ll}\ \displaystyle{\frac{1}{C\ve^{\gamma}}}, &\quad \textrm{if } \gamma<1 \\
 \displaystyle{ \frac{1}{C\ve |\log \ve|}}, & \quad \textrm{if } \gamma=1 \\
\displaystyle{  \frac{1}{C\ve}}, & \quad \textrm{if } \gamma>1. \end{array} \right.
\end{equation}
\end{theorem}

The statement and proof of Theorem \ref{thmoptimal} appear to be new even in the case when $D_1$ and $D_2$ are strictly convex, since the optimality results of \cite{BLY} made symmetry assumptions on the sets $D_1, D_2$ and $\Omega$.  

We have assumed smoothness of $\partial \Omega$, $\partial D_1$,  $\partial D_2$ and $\varphi$ for the sake of simplicity.  As in \cite{BLY} the regularity can be relaxed to $C^{2+\beta}$ for the boundaries (for some $0<\beta<1$), and $C^2$ for $\varphi$.

The outline of this paper is as follows.  In Section \ref{sectionprelim} we recall some preliminary results from \cite{BLY} which we will need to make use of.  In Sections \ref{sectionproof} and \ref{sec: opt bds} we prove Theorems \ref{main thm} and \ref{thmoptimal} respectively.  Finally, in Section \ref{sectionexamples} we illustrate our results with some examples.

\section{Preliminaries} \label{sectionprelim}

In this section, we gather some preliminary results whose proofs follow from the corresponding arguments of Bao-Li-Yin \cite{BLY}.  First, the bound on $|\nabla u|$ reduces to an estimate on $|C_1-C_2|$, where we recall that $u=C_1$ on $D_1$ and $u=C_2$ on $D_2$.  

\begin{lemma}
\label{lemma: bds for grad u in terms of C1-C2}
There exists a constant $C$ independent of $\varepsilon$ such that
\begin{equation}
  \notag 
  \frac { | C_1 - C_2 | } { \varepsilon }
  \le \sup_{\tilde{\Omega}} |\nabla u| \le
  \frac C { \varepsilon } | C_1 - C_2 | + C.
\end{equation}
\end{lemma}

\begin{proof}  See the proof of \cite[Proposition 2.1]{BLY}.  For the reader's convenience, we sketch the idea here.  

The lower bound of $\sup|\nabla u|$ follows from the Mean Value Theorem.  For the upper bound, we write $u = v + w + C_2$ where $v$ and $w$ are the unique solutions of 
\begin{equation}
  \notag 
  \begin{aligned}
    \Delta v & = 0 \text{ in } \tilde\Omega, \quad 
    v  = C_1 - C_2 \text{ on } \partial D_1, \quad 
    v  = 0 \text{ on } \partial D_2 \cup \partial \Omega \\
     \Delta w & = 0 \text{ in } \tilde\Omega, \quad 
    w  = 0 \text{ on } \partial D_1 \cup \partial D_2, \quad 
    w  = \varphi - C_2 \text{ on } \partial \Omega.
      \end{aligned}
\end{equation}
Since $|v|\le |C_1-C_2|$ we can apply standard gradient estimates for harmonic functions to obtain $|\nabla v|\le C|C_1-C_2|/\ve$.  On the other hand $|\nabla w|\le C$ on $\partial D_1$ by comparison with a harmonic function on $\Omega\setminus \overline{D}_1$ which vanishes on $\partial D_1$ and has uniformly large absolute value on $\partial \Omega$.  Similarly $|\nabla w|$ is bounded on $\partial D_2$, and  hence  $|\nabla w|\le C$ on $\tilde{\Omega}$.   The upper bound of $|\nabla u|$ follows.
\end{proof}

For the  second result of this section, we need some definitions.  Define functions $v_1$ and $v_2$ by
\begin{equation*}
  \notag 
  \begin{aligned}
  \Delta v_1 & = 0 \text{ in } \tilde\Omega, \quad 
  v_1  = 0 \text{ on } \partial D_2 \cup \partial \Omega, \quad 
  v_1  = 1 \text{ on } \partial D_1 \\
    \Delta v_2 & = 0 \text{ in } \tilde\Omega, \quad 
  v_2  = 0 \text{ on } \partial D_1 \cup \partial \Omega, \quad
  v_2  = 1 \text{ on } \partial D_2
  \end{aligned}
\end{equation*}
As in \cite{BLY}, define the linear functional
\begin{equation}
  \notag 
  Q_{\varepsilon}(\varphi) :=
  \int_{\partial \Omega} \varphi
  \frac { \partial v_1 } { \partial \nu }
  \int_{\partial \Omega}
  \frac { \partial v_2 } { \partial \nu }
  - \int_{\partial \Omega} \varphi
  \frac { \partial v_2 } { \partial \nu }
  \int_{\partial \Omega}
  \frac { \partial v_1 } { \partial \nu }.
\end{equation}
The following gives an estimate for $|C_1-C_2|$.

\begin{lemma}
\label{lemma: bds for |C1-C2|} \ 
Assume there is a uniform constant $c>0$ such that 
\begin{equation}\label{ass12}
- \int_{\partial \Omega} \frac{\partial v_1}{\partial \nu} \ge c,
\quad
- \int_{\partial \Omega} \frac{\partial v_2}{\partial \nu} \ge c.
\end{equation}
Then 
\begin{equation}
  \notag 
c |Q_{\ve}(\varphi)| \left( \int_{\tilde{\Omega}} |\nabla v_1|^2 \right)^{-1} \le 
| C_1 - C_2 |
  \le C |Q_{\ve}(\varphi)| \left( \int_{\tilde{\Omega}} |\nabla v_1|^2 \right)^{-1}.
\end{equation}
\end{lemma}

\begin{proof}   The proof is contained in \cite[Section 2]{BLY}, but again for the sake of convenience we include here a brief outline of the argument.  Define
\[
\begin{split}
a_{11} := {} & -\int_{\partial D_1} \frac{\partial v_1}{\partial \nu} = \int_{\tilde{\Omega}} |\nabla v_1|^2, \quad a_{22}:=  -\int_{\partial D_2} \frac{\partial v_2}{\partial \nu} = \int_{\tilde{\Omega}} |\nabla v_2|^2, \\
a_{12}: = {} & -\int_{\partial D_1} \frac{\partial v_2}{\partial \nu} = \int_{\tilde{\Omega}} \nabla v_1 \cdot \nabla v_2 = -\int_{\partial D_2} \frac{\partial v_1}{\partial \nu}=: a_{21}, \\
\end{split}
\]
where we have used integration by parts.  Define another function $v_3$ by
\[
 \begin{aligned}
  \Delta v_3 & = 0 \text{ in } \tilde\Omega, \quad 
  v_3  = 0 \text{ on } \partial D_1 \cup  \partial D_2, \quad 
  v_3  = \varphi \text{ on }  \partial \Omega,
  \end{aligned}
\]
and for $i=1,2$,
$$b_i := - \int_{\partial D_i} \frac{\partial v_3}{\partial \nu}= \int_{\tilde{\Omega}} \nabla v_i \cdot \nabla v_3 = \int_{\partial \Omega} \varphi \frac{\partial v_i}{\partial \nu}.$$

Since $u=C_1v_1+C_2v_2+v_3$, the fourth line of
(\ref{perf cond prob}) gives
\[
  \notag 
  \begin{aligned}
  a_{11} C_1 + a_{12} C_2 + b_1 & = 0, \quad 
  a_{21} C_1 + a_{22} C_2 + b_2  = 0.
  \end{aligned}
\]
Hence
\begin{equation} 
  \label{eqn for C1-C2}
  C_1 - C_2
  = \frac
  { ( a_{11} + a_{21} ) b_2 - ( a_{22} + a_{12} ) b_1 }
  { a_{11} a_{22} - a_{12}^2 } = \frac{Q_{\ve}(\varphi)}{ a_{11} a_{22} - a_{12}^2 },
\end{equation}
where we have used the fact that 
$$a_{11}+a_{21} = -\int_{\partial \Omega} \frac{\partial v_1}{\partial \nu}, \quad a_{22}+a_{12}=-\int_{\partial \Omega} \frac{\partial v_2}{\partial \nu},$$
and assuming that $a_{11}a_{22} - a_{12}^2 \neq 0$, which we will shortly prove.

From the assumption (\ref{ass12}) and standard derivative estimates for harmonic functions
$v_1$, $v_2$ in a neighborhood of the boundary $\Omega$, we have
$$c \le a_{11} + a_{21} \le C, \quad c \le a_{22} + a_{12} \le C.$$
Next
$$a_{11}a_{22} - a_{12}^2 = a_{11} (a_{22}+a_{12}) - a_{12} (a_{11}+a_{21}),$$ 
and hence, since $a_{12}\le 0$, so in particular, $|a_{12}|<a_{11}$  (using $a_{11}+a_{12}> 0$), 
\begin{equation} \label{det}
c a_{11} \le a_{11}a_{22} -a_{12}^2 \le Ca_{11}.
\end{equation} 
The result follows from (\ref{eqn for C1-C2}) and (\ref{det}).
\end{proof}

\section{Proof of Theorem \ref{main thm}}
\label{sectionproof}

In this section we complete the proof of the main theorem.

First we note that the assumption $(*)$ in the introduction implies that 
there exists a connected domain $W_1$ in $\Omega \setminus (D^0_1 \cup D^0_2)$ with smooth boundary $\partial W_1$ such  that $\partial W_1$ has an open portion on $\partial \Omega$ and another open portion on $\partial D^0_1$.  Moreover, we may assume that $\overline{W}_1$ and $\overline{D^0_2}$ are disjoint.  Similarly there exists another domain $W_2$, interchanging the roles of $D^0_1$ and $D^0_2$.
See figure \ref{fig: star assumption} for an example illustrating
this.


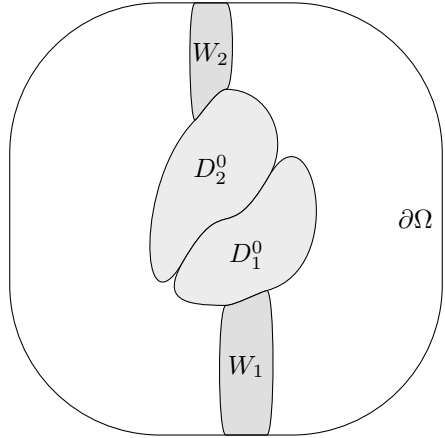
\begin{figure}[h]
\begin{tikzpicture}
[
  regionD1/.style={fill=gray!15},
  regionD2/.style={fill=gray!15},
  regionW/.style={fill=gray!25},
  scale=1.15,
]

\path
  (-0.5,-0.5) coordinate (A)
  (0.5,0.5)   coordinate (B)
  (0,1.5)     coordinate (C)
  +(225:0.5)  coordinate (D)
  (0,-1)      coordinate (E)
  +(20:0.5)   coordinate (F)
;

\def\pathAB{(A) .. controls +(60:1) and +(-120:1) .. (B)}
\def\pathBA{(B) .. controls +(-120:1) and +(60:1) .. (A)}
\def\pathCD{(C) .. controls +(west:0.5) and +(north:0.5) .. (D)}
\def\pathEF{(E) .. controls +(south:0.5) and +(south:0.5) ..  (F)}

\filldraw [regionD2]
  \pathAB
  .. controls +(60:0.5) and +(east:0.5) ..
  (C)
  .. controls +(west:0.1) and +(north east:0.1) ..
  (D)
  .. controls +(south west:1) and +(-120:1) ..
  cycle
;
\coordinate (X2) at (D |- 0,2.5);
\coordinate (Y2) at (C |- 0,2.5);

\filldraw [regionW]
  (C)
  .. controls +(east:0.1) and +(east:0.1) ..
  (Y2) -- (X2)
  .. controls +(west:0.1) and +(south west:0.1) ..
  (D)
  .. controls +(north east:0.1) and +(west:0.1) ..
  cycle
;
\node at ($ ($ (C)!0.5!(D) $)!0.5!($ (X2)!0.5!(Y2) $) $) {$W_2$};

\filldraw [regionD1]
  \pathBA
  .. controls +(-120:0.5) and +(west:0.5) ..
  (E)
  .. controls +(east:0.1) and +(-180+15:0.1) ..
  (F)
  .. controls +(15:1) and +(60:1) ..
  cycle
;
\node at ($ (0,0)!0.5!(F) $) {$D_1^0$};
\node at ($ (0,0)!0.5!(D) $) {$D_2^0$};

\coordinate (X1) at (E |- 0,-2.5);
\coordinate (Y1) at (F |- 0,-2.5);

\filldraw [regionW]
  (E)
  .. controls +(east:0.1) and +(-180+15:0.1) ..
  (F)
  .. controls +(15:0.1) and +(east:0.1) ..
  (Y1) -- (X1)
  .. controls +(west:0.1) and +(west:0.1) ..
  cycle
;
\node at ($ ($ (E)!0.5!(F) $)!0.5!($ (X1)!0.5!(Y1) $) $) {$W_1$}; 

\def\bdryOmega%
{[rounded corners=2cm] (-2.5,-2.5) rectangle (2.5,2.5)}
\draw \bdryOmega;

\node at (2.5,0) [left] {$\partial \Omega$};


\end{tikzpicture}
\caption{Example configuration satisfying $(*)$.}
\label{fig: star assumption}
\end{figure}


In order to apply Lemma \ref{lemma: bds for |C1-C2|} we need to establish the estimates (\ref{ass12}).

\begin{lemma}\label{proveass}
There is a uniform constant $c>0$ such that 
\begin{equation}\label{ass12e}
-\int_{\partial \Omega} \frac{\partial v_1}{\partial \nu} \ge c, \quad -\int_{\partial \Omega} \frac{\partial v_2}{\partial \nu} \ge c.
\end{equation}
\end{lemma}
\begin{proof}  Note that by the definitions of $v_1$ and $v_2$ we have $\frac{\partial v_1}{\partial \nu}, \frac{\partial v_2}{\partial \nu} \le 0$ on $\partial \Omega$.  We will show that $\left|\frac{\partial v_1}{\partial \nu}\right|$ and $\left|\frac{\partial v_2}{\partial \nu}\right|$ must be bounded uniformly away from zero on a portion of the boundary $\partial \Omega$.

We make use of a {\bf basic fact} (see for example \cite[Ex. 2.2]{GT}) that if a function $w$ is harmonic on a connected domain $D$ and has $w=0=\frac{\partial w}{\partial \nu}$ on an open smooth portion of the boundary $\partial D$ then $w$ vanishes identically on $D$.

Let $\Pi$ be an open portion of $\partial \Omega$ which is contained in $\partial W_1$. We recall that $v_1=v_1^{\ve}$ depends on $\ve$.
We claim that for $\ve>0$ sufficiently small, $\sup_{p \in \Pi} |\frac{\partial v_1^{\ve}}{\partial\nu}(p)|  \ge c$ for some constant $c>0$.  Indeed if not then we can find a sequence $\ve_j\rightarrow 0$ with 
\begin{equation} \label{supPi}
\sup_{p \in \Pi} \left| \frac{\partial v_1^{\ve_j}}{\partial\nu}(p) \right| \le 1/j \rightarrow 0, \textrm{ as } j \rightarrow \infty.
\end{equation} The functions $v_1^{\ve_j}$ are not necessarily defined on $W_1$ since $D_1^{\ve}$ is changing with $\ve$.  However, after composing with diffeomorphisms which converge to the identity, and are equal to the identity in a neighborhood of $\partial W_1 \cap \partial \Omega$, and after passing to a subsequence, the functions $v_1^{\ve_j}$ converge smoothly on $W_1$ to a harmonic function $v_1^0$ on $W_1$ which is equal to $0$ on $\partial \Omega \cap \partial W_1$ and equal to $1$ on $\partial D^0_1 \cap \partial W_1$.

In particular the function $v_1^0$ cannot be identically zero on $W_1$.  But from (\ref{supPi}) we obtain $\frac{\partial v_1^0}{\partial \nu}=0$ on an open portion of the boundary of $\partial W_1$, a contradiction by the basic fact.

Hence we have shown that $\sup_{p \in \Pi} |\frac{\partial v_1^{\ve}}{\partial\nu}(p)|  \ge c$  for $\ve>0$ sufficiently small and it follows by standard derivative estimates for $v^{\ve}_1$ that $|\frac{\partial v_1^{\ve}}{\partial\nu}(p)|  \ge c/2$ on a small open portion of $\partial \Omega$.  This establishes the required estimate for $v_1$.  The argument for $v_2$ is similar.
\end{proof}

Note that without assumption $(*)$ the estimates (\ref{ass12e}) may fail.   Indeed, if $D_2$ completely surrounds $D_1$ (for example if $D_1$ is a solid ball and $D_2$ is a solid spherical
shell enclosing $D_1$) then $v_1 \equiv 0$ 
 on the region outside the outer boundary of
$D_2$ and hence $\frac{\partial v_1}{\partial \nu}=0$ on $\partial \Omega$.  We wish to exclude this case, which is not very interesting from the point of view of estimating $|\nabla u|$.
Physically,
if $u$ represents an electric potential and 
$D_1, D_2$ are perfect conductors then $u$ will take
a constant value $C_2 = C_1$ throughout $D_1, D_2$ and the
region between the two.  In this case  $\sup|\nabla u|$ does not blow up as $\ve \rightarrow 0$.

The proof of Theorem \ref{main thm} will now be an almost immediate consequence of the following lemma, whose proof uses the same basic strategy as in \cite{BLY}.  A key difference from \cite{BLY} is that we use a patching argument to deal with multiple coordinate boxes, rather than working in a single one.  Also, the fact that the exponents $\alpha^i_1, \ldots, \alpha^i_n$ are not necessarily equal gives rise to a more complicated integral.

\begin{lemma} \label{lemmakey}
Let $v_1$ satisfy
\[
  \begin{aligned}
  \Delta v_1 & = 0 \text{ in } \tilde\Omega, \quad 
  v_1  = 0 \text{ on } \partial D_2 \cup \partial \Omega, \quad 
  v_1  = 1 \text{ on } \partial D_1.
  \end{aligned}
\]
Then 
\begin{equation}
  \notag 
  \begin{alignedat}{3}
  \frac 1 C
  & < \int_{\tilde \Omega} | \nabla v_1 |^2
  < C,
  & \qquad \text{if } \gamma > 1
  \\
  \frac 1 C \log \frac 1 {\varepsilon}
  & < \int_{\tilde \Omega} | \nabla v_1 |^2
  < C \log \frac 1 {\varepsilon},
  & \qquad \text{if } \gamma = 1
  \\
  \frac 1 C \frac 1 {\varepsilon^{1-\gamma}}
  & < \int_{\tilde \Omega} | \nabla v_1 |^2
  < C \frac 1 {\varepsilon^{1-\gamma}},
  & \qquad \text{if } \gamma < 1,
  \end{alignedat}
\end{equation}
where we recall that $\gamma$ is defined by (\ref{gamma def}).
\end{lemma}

\begin{proof}
Recall that for each $i = 1, \ldots, k$ the
boundaries of $D_1, D_2$ are described in the box $U_i$ as the
graphs of the functions $f_i(x')$ and $g_i(x')$,
with $f_i < g_i$.  

For each $i$ define the function
\begin{equation}
  \label{def w sub i}
  w_i(x) = \frac%
  {g_i(x') - x_n}%
  {g_i(x') - f_i(x')}
\end{equation}
which is defined on all of
$U_i \cap \tilde\Omega$.  

For the lower bound of $\int_{\tilde{\Omega}}|\nabla v_1|^2$, we fix an index $i$ and work in  $U_i$.
For $x' \in Q_r$ the
function $x_n \mapsto w_i(x', x_n)$ is of the form
$x_n \mapsto a(x') x_n + b(x')$
and has the property that
$w_i|_{x_n=f_i(x')} = 1 = v_1|_{x_n=f_i(x')}$
and
$w_i|_{x_n=g_i(x')} = 0 = v_1|_{x_n=g_i(x')}$.  
In particular,
for any fixed $x' \in Q_r$,
\begin{equation}
  \notag 
  \int_{f_i(x')}^{g_i(x')} | \partial_{x_n} w_i |^2 \, dx_n
  \le
  \int_{f_i(x')}^{g_i(x')} | \partial_{x_n} v_1 |^2 \, dx_n
  \le
  \int_{f_i(x')}^{g_i(x')} | \nabla v_1 |^2 \, dx_n, 
\end{equation}
where the first inequality follows from the fact that linear functions of one variable minimize the Dirichlet energy among functions with the same endpoints.
Therefore
\begin{equation}
  \label{lwr bd v1 calc}
  \begin{aligned}
    \int_{\tilde\Omega} | \nabla v_1 |^2
    & \ge \int_{U_i \cap \tilde\Omega} | \nabla v_1 |^2 \\
    & = \int_{x' \in Q_r} \int_{f_i(x')}^{g_i(x')}
      | \nabla v_1 |^2 \, dx_n \, dx' \\
    & \ge \int_{x' \in Q_r} \int_{f_i(x')}^{g_i(x')}
      | \partial_{x_n} w_i |^2 \, dx_n \, dx' \\
    & = \int_{x' \in Q_r}
      \frac {dx'} {g_i(x') - f_i(x')} \\
    & \ge \frac 1 C \int_{x' \in Q_r}
      \frac{dx'}{ \varepsilon + \sum_{j=1}^{n-1} x_j^{2 \alpha^i_j} },
  \end{aligned}
\end{equation}
recalling (\ref{fi and gi}). 
Define
\begin{equation}
  \label{Ii}
  I_i(\varepsilon) := \int_{x' \in Q_r}
  \frac{dx'}{ \varepsilon + \sum_{j=1}^{n-1} x_j^{2 \alpha^i_j} } . 
\end{equation}
Then we see that 
$\int_{\tilde\Omega} | \nabla v_1 |^2
\ge \frac 1C \max_i I_i(\varepsilon)$.  
Below we will bound the term $I_i(\varepsilon)$.  
Before then we consider the upper bound.

Note that by standard estimates and the definition
of $V \subset \tilde\Omega$
we may assume that 
$|\nabla v_1|^2 \le C$ at all points of $V \subset \tilde\Omega$.  
In particular $\int_{V} |\nabla v_1|^2 < C$.           

Recall that the region $\tilde{\Omega} \setminus V$ is covered by the ``half sized'' boxes $U^{(r/2)}_i$ of size $r/2$.  We denote by $U^{(3r/4)}_i$ the ``three-quarter sized'' boxes of size $3r/4$.  Let $\{\sigma_i\}_{i=1}^k$ be a partition of unity such that: (1) each $\sigma_i: \bigcup_j U_j \rightarrow [0,1]$ is a smooth function with compact support in $U_i=U^{(r)}_i$; and (2) $\sum_i \sigma_i=1$ on $\bigcup_i U^{(3r/4)}_i$.  The function $$w:= \sum_{i=1}^k \sigma_i w_i,$$ is a well-defined smooth function on $\bigcup_{i=1}^k U_i$ which is equal to $1$ on $\partial D_1 \cap \bigcup_i U^{(3r/4)}_i$ and equal to $0$ on $\partial D_2 \cap \bigcup_i U^{(3r/4)}_i$. 

Next let
$\rho: \Omega \rightarrow [0,1]$
be a smooth  cut-off function which is identically equal to $1$ on the union of half-sized boxes $\bigcup_i U_i^{(r/2)}$ and is supported on the union of three-quarter sized boxes $\bigcup_i U^{(3r/4)}_i$.

Then the function $\displaystyle{W=\rho w + (1-\rho) v_1}$ has the following properties.

\begin{enumerate}
\item[(i)] $W$ is a well-defined continuous function on $\overline{\tilde{\Omega}}$, smooth on $\tilde{\Omega}$.
\item[(ii)] $W$ is equal to $1$ on $\partial D_1$ and equal to $0$ on $\partial D_2 \cup \partial \Omega$.
\end{enumerate}

For (ii), we observe that at points in $\bigcup_i U_i^{(3r/4)}$ the functions $w$ and $v_1$ are both equal to $1$ on $\partial D_1$ and equal to $0$ on $\partial D_2$, whereas outside this union, $W=v_1$ which is equal to $1$ on $\partial D_1$ and vanishes on $\partial D_2 \cup \partial \Omega$. 

Note also that $0 \le v_1, w, w_i, \sigma_i, \rho \le 1$ at all
points of $\tilde{\Omega}$ that each is defined,
and that $| \nabla \sigma_i |, | \nabla \rho | \le C$.

Since $v$ is harmonic on $\tilde{\Omega}$, conditions (i) and (ii) imply that $\int_{\tilde{\Omega}}|\nabla v_1|^2 \le \int_{\tilde{\Omega}}|\nabla W|^2$ and hence
\begin{equation}
  \label{upper bd calc for v1}
  \begin{aligned}
    \int_{\tilde\Omega} |\nabla v_1|^2
    & \le \int_{\tilde\Omega} |\nabla (\rho w + (1 - \rho) v_1)|^2 \\
    & = \int_{\tilde\Omega}
      | w \nabla \rho
      + \rho \nabla w
      - v_1 \nabla \rho
      + (1-\rho) \nabla v_1 |^2 \\
    & \le C \bigl(
      1 + \int_{\cup_{i=1}^k U_i  \cap \tilde\Omega} |\nabla w|^2
      \bigr) \\
    & \le C \bigl(
      1 + \sum_{i=1}^k \int_{U_i \cap \tilde\Omega}
      | \nabla w_i |^2 \bigr),
  \end{aligned}
\end{equation}
where for the third line we used the fact that $|\nabla v_1|$ is uniformly bounded on the set $V$ and hence on $\tilde{\Omega} \setminus \bigcup_i U_i^{(r/2)}$.
From (\ref{def w sub i}) we estimate on $U_i \cap \tilde{\Omega}$,
\begin{equation}
  \notag 
  | \nabla w_i |^2(x)
  \le
  \frac
  { C } 
  { ( g_i(x') - f_i(x') )^2 }, 
\end{equation}
and hence for each $i = 1, \ldots, k$,
\begin{equation}
\label{up bd for int of deriv of wi squared}
  \begin{aligned}
  \int_{U_i \cap \tilde \Omega}
  | \nabla w_i |^2
  & \le \int_{x' \in Q_r} \int_{f_i(x')}^{g_i(x')}
  \frac { C\,dx_n } { (g_i(x') - f_i(x'))^2 }\,dx'
  \\
  & = C \int_{x' \in Q_r} \frac { dx' } { g_i(x') - f_i(x') }
  \\
  & \le C I_i(\varepsilon) . 
  \end{aligned}
\end{equation}
Combining
(\ref{upper bd calc for v1})
and 
(\ref{up bd for int of deriv of wi squared})
we find 
\begin{equation}
  \label{upper bd of int of sq of deriv of v1 summary}
  \int_{\tilde\Omega} | \nabla v_1 |^2
  \le C\left( 1+
  \max_{i = 1, \ldots, k} I_i(\varepsilon) \right).
\end{equation}

The lemma is then a consequence of (\ref{lwr bd v1 calc}), (\ref{upper bd of int of sq of deriv of v1 summary}) and the following elementary claim.

\medskip
\noindent
{\bf Claim.} \ For $i=1, \ldots, k$, writing $\displaystyle{\gamma_i = \sum_{j=1}^{n-1} \frac{1}{2\alpha_j^i}}$, we have
\begin{equation}
  \label{bds for I}
  \begin{alignedat}{3}
  \frac 1 C
  & < I_i({\ve})
  < C,
  & \qquad \text{if } \gamma_i > 1
  \\
  \frac 1 C \log \frac 1 {\varepsilon}
  & < I_i({\ve})
  < C \log \frac 1 {\varepsilon},
  & \qquad \text{if } \gamma_i = 1
  \\
  \frac 1 C \frac 1 {\varepsilon^{1-\gamma_i}}
  & < I_i({\ve})
  < C \frac 1 {\varepsilon^{1-\gamma_i}},
  & \qquad \text{if } \gamma_i < 1,
  \end{alignedat}
\end{equation}
where we recall that $I_i(\ve)$ is defined by (\ref{Ii}).

\medskip
To prove the claim, we drop the index $i$ and write
 $\alpha = (\alpha_1, \ldots, \alpha_{n-1})$.
Rearranging the components if necessary we
assume that $1 \le \alpha_1, \ldots, \alpha_{\ell} < \infty$ 
and $\alpha_{\ell+1} = \ldots = \alpha_{n-1} = \infty$.
Then we need to compute the integral
\begin{equation}
\label{rewrite the integral}
  I(\varepsilon)
  = \int_{x' \in Q_r}
  \frac
  { dx' }
  { \varepsilon + \sum_{j=1}^{n-1} x_j^{2\alpha_j} }
  = 2^{n-1} r^{n-1-\ell}
  \underbrace{\int_0^r \cdots \int_0^r}_{\ell}
  \frac
  { dx_1\,\cdots\,dx_{\ell} }
  { \varepsilon + \sum_{j=1}^{\ell} x_j^{2\alpha_j} } .
\end{equation}
Note that $\ell = 0$
if and only if $\alpha = ( \infty, \ldots, \infty )$.
In this case the above integral evaluates
exactly to $2^{n-1}r^{n-1}/\ve$, giving (\ref{bds for I}) in the case $\gamma=\gamma_i=0$.  Henceforth we  assume $\ell \ge 1$.

We first reduce the claim to estimating an integral of one variable.  Namely, we will show that
\begin{equation}
\label{claimse}
  \frac 1 C \int_0^{R_0}
  \frac
  { \rho^{2 \gamma - 1} \, d\rho }
  { \varepsilon + \rho^2 }
  \le I(\varepsilon)
  \le C \int_0^{R_1}
  \frac
  { \rho^{2 \gamma - 1} \, d\rho }
  { \varepsilon + \rho^2 },
\end{equation}
where $R_0 \coloneqq \min_j r^{\alpha_j}$,
$R_1 \coloneqq \sqrt\ell \max_j r^{\alpha_j}$ (for $j$ ranging from $1$ to $\ell$) and
$C>0$ depends only on $\alpha$, $r$ and $n$.

Making the substitution $u_j = x_j^{\alpha_j}$ we find
\begin{equation}
\label{rewrite the int pt 2}
  \begin{aligned}
  \int_0^r \cdots \int_0^r
  \frac
  { dx_1\,\cdots\,dx_{\ell} }
  { \varepsilon + \sum_{j=1}^{\ell} x_j^{2\alpha_j} }
  & =
  \frac 1 {\alpha_1\cdots \alpha_{\ell}}
  \int_0^{r^{\alpha_1}} \cdots \int_0^{r^{\alpha_{\ell}}}
  \frac
  { \prod_{j} u_j^{ \frac 1 {\alpha_j} - 1 } du_1\,\cdots\,du_{\ell} }
  { \varepsilon + \sum_{j=1}^{\ell} u_j^{2} }
  \\
  & < \frac 1 {\alpha_1 \cdots \alpha_{\ell}}
  \int_{B_{R_1}^+}
  \frac
  { \prod_{j} u_j^{ \frac 1 {\alpha_j} - 1 } du_1\,\cdots\,du_{\ell} }
  { \varepsilon + \sum_{j=1}^{\ell} u_j^{2} },
  \end{aligned}
\end{equation}
where $B_{R_1}^+$ denotes the portion of the ball of radius
$R_1 = \sqrt\ell \max_j r^{\alpha_j}$, 
centered at the origin 
in $\mathbb{R}^\ell$, where all the coordinates are positive.
Similarly we find
\begin{equation}
\label{rewrite the int pt 2 lwr bd}
  \begin{aligned}
  \int_0^r \cdots \int_0^r
  \frac
  { dx_1\,\cdots\,dx_{\ell} }
  { \varepsilon + \sum_{j=1}^{\ell} x_j^{2\alpha_j} }
  & > \frac 1 {\alpha_1 \cdots \alpha_{\ell}}
  \int_{B_{R_0}^+}
  \frac
  { \prod_{j} u_j^{ \frac 1 {\alpha_j} - 1 } du_1\,\cdots\,du_{\ell} }
  { \varepsilon + \sum_{j=1}^{\ell} u_j^{2} },
  \end{aligned}
\end{equation}
where $R_0 = \min_j r^{\alpha_j}$.  
In spherical coordinates $(\rho, \varphi_1, \ldots, \varphi_{\ell-1})$ we have
\begin{equation}
  \notag 
  \begin{aligned}
  u_j & = \rho ( \cos \varphi_j )
        \bigl( \prod_{q<j} \sin \varphi_q \bigr)
        \quad \text{ for } j = 1,\dots, \ell-1 \\
  u_{\ell} & = \rho \bigl( \prod_{q<\ell} \sin \varphi_q \bigr) \\
  du & = \rho^{\ell-1} \prod_{j=1}^{\ell-2}
       (\sin \varphi_j)^{\ell-1-j}
       d \rho \, d \varphi_1 \, \cdots \, d \varphi_{\ell-1},
  \end{aligned}
\end{equation}
and we find for any $R>0$ that 
\begin{equation}
\label{int switch to spherical}
  \int_{B_R^+}
  \frac
  { \prod_{j} u_j^{ \frac 1 {\alpha_j} - 1 } du_1\,\cdots\,du_{\ell} }
  { \varepsilon + \sum_{j=1}^{\ell} u_j^{2} }
  =
  A
  \int_0^R
  \frac
  { \rho^{2\gamma - 1} \, d\rho }
  { \varepsilon + \rho^2 },
\end{equation}  
where
$ \gamma = \sum_j \frac 1 {2 \alpha_j} $ 
and 
\begin{equation}
\label{A sub m}
  A =
  \prod_{q=1}^{\ell-1}
  \int_0^{\frac \pi 2}
  (\sin \varphi_q )^{-1 + \sum_{j=q+1}^{\ell} \frac 1 {\alpha_j}}
  (\cos \varphi_q )^{-1 +  \frac 1 {\alpha_q}} \,
  d \varphi_q .
\end{equation}
Combining
(\ref{rewrite the integral}),
(\ref{rewrite the int pt 2}),
(\ref{rewrite the int pt 2 lwr bd}),
 (\ref{int switch to spherical}) and (\ref{A sub m}) proves (\ref{claimse}) as required.

We can now finish the proof of the claim.   If $\gamma > 1$ then 
\begin{equation}
  \notag 
  \int_0^{R_1}
  \frac
  { \rho^{2 \gamma -1} }
  { \varepsilon + \rho^2 }
  \, d\rho
  < \int_0^{R_1} \rho^{2 \gamma - 3} \, d\rho
  = C,
\end{equation}
and
\begin{equation}
  \notag 
  \int_0^{R_0}
  \frac
  { \rho^{2 \gamma -1} }
  { \varepsilon + \rho^2 }
  \, d\rho
  > \int_{R_0/2}^{R_0}
  \frac { \rho^{2\gamma-1} } { 1 + \rho^2 }d\rho
  = \frac1C 
\end{equation}
for all $\varepsilon < 1$,
which establishes this case.

Next, if $\gamma = 1$, then
\begin{equation}
  \notag 
  \int_0^{R_1}
  \frac
  { \rho^{2 \gamma -1} }
  { \varepsilon + \rho^2 }
  \, d\rho
  = \frac 12 \log
  \bigl( 1 + \frac {R_1^2} {\varepsilon} \bigr)
  < C \log \frac 1 {\varepsilon} 
\end{equation}
and
\begin{equation}
  \notag 
  \int_0^{R_0}
  \frac
  { \rho^{2 \gamma -1} }
  { \varepsilon + \rho^2 }
  \, d\rho
  = \frac 12 \log
  \bigl( 1 + \frac {R_0^2} {\varepsilon} \bigr)
  > \frac 1 C \log \frac 1 {\varepsilon} .  
\end{equation}

Finally if $0<\gamma < 1$
let $\beta = 2 \gamma - 1$
and note that $-1 < \beta < 1$.
Then by setting $v = \rho/\varepsilon^{1/2}$ we see that 
for all $R>0$ we have
\begin{equation}
  \notag 
  \int_0^R \frac { \rho^\beta \,d\rho } { \varepsilon + \rho^2 }
  =
  \frac 1 {\varepsilon^{ 1 - \gamma }}
  \int_0^{R/\varepsilon^{1/2}}
  \frac { v^\beta \, dv } { 1 + v^2 } .
\end{equation}
But since $-1 < \beta < 1$ we see that
\begin{equation}
  \notag 
  \int_0^{R_1/\varepsilon^{1/2}}
  \frac { v^\beta \, dv } { 1 + v^2 }
  <
  \int_0^1 v^\beta \, dv
  + \int_1^{\infty} v^{\beta-2} \, dv
  = C
\end{equation}
and
\begin{equation}
  \notag 
  \int_0^{R_0/\varepsilon^{1/2}}
  \frac { v^\beta \, dv } { 1 + v^2 }
  >
  \frac 12 \int_0^1 v^\beta \, dv
  = \frac 1 C,
\end{equation}
for all $\varepsilon < R_0^2$.  
This completes the proof of the claim and hence the lemma.
\end{proof}

Finally we complete the proof of the main theorem.

\begin{proof}[Proof of Theorem \ref{main thm}]
This is an immediate consequence of Lemmas \ref{lemma: bds for grad u in terms of C1-C2}, \ref{lemma: bds for |C1-C2|}, \ref{proveass} and \ref{lemmakey} 
and the fact that $|Q_{\ve}(\varphi)|\le C$.
\end{proof}

\section{Optimality of the bounds}
\label{sec: opt bds}

In this section we prove Theorem \ref{thmoptimal} on the optimality of the bounds of Theorem \ref{main thm}.
From  Lemmas \ref{lemma: bds for grad u in terms of C1-C2}, \ref{lemma: bds for |C1-C2|},  \ref{proveass} and \ref{lemmakey} we have
\begin{equation}
  \notag 
 \sup_{\tilde{\Omega}} |\nabla u| \ge \left\{ \begin{array}{ll}\ \displaystyle{\frac{|Q_{\varphi}(\varphi)|}{C\ve^{\gamma}}}, &\quad \textrm{if } \gamma<1 \\
 \displaystyle{ \frac{|Q_{\ve}(\varphi)|}{C\ve |\log \ve|}}, & \quad \textrm{if } \gamma=1 \\
\displaystyle{  \frac{|Q_{\ve}(\varphi)|}{C\ve}}, & \quad \textrm{if } \gamma>1. \end{array} \right.
\end{equation}
Hence to prove Theorem \ref{thmoptimal} it suffices to show the existence of $\varphi$ so that $|Q_{\ve}(\varphi)|  \ge c$ for a constant $c>0$ independent of $\ve$.

Note that in general, it is not the case that $Q_{\ve}(\varphi)$ will be nonzero for all choices of $\varphi$ (one could take $\varphi=0$, for example).  

\begin{proof}[Proof of Theorem \ref{thmoptimal}]
First we note that there exists an  open portion $P$ of $\partial \Omega$ such that  
\begin{equation}\label{claimeqn}
\displaystyle{\left| \frac{\partial v^{\ve}_1}{\partial \nu} - \frac{\partial v^{\ve}_2}{\partial \nu}\right| \ge c},
\end{equation} for all $\ve>0$ sufficiently small.

Indeed, to see this 
define $v^{\ve}=v^{\ve}_1-v^{\ve}_2$, which is harmonic on $\Omega \setminus (D^{\ve}_1 \cup D^{\ve}_2)$, vanishes on $\partial \Omega$, is equal to $1$ on $D_1^{\ve}$ and equal to $-1$ on $D_2^{\ve}$.  We apply the same argument as in Lemma \ref{proveass} above. Let $W_1$ be as defined there, and 
let $\Pi$ be an open portion of $\partial \Omega$ which is contained in $\partial W_1$.  For $\ve>0$ sufficiently small, we claim that $\sup_{p \in \Pi} |\frac{\partial v^{\ve}}{\partial\nu}(p)|  \ge c$ for some constant $c>0$.  If not then we can find a sequence $\ve_j \rightarrow 0$ such that $v^{\ve_j}$ converges (after diffeomorphisms) to a harmonic function $v^0$ on $W_1$ which is equal to $0$ on $\partial \Omega \cap \partial W_1$, equal to $1$ on $\partial D^0_1 \cap \partial W_1$ and 
$\frac{\partial v^0}{\partial \nu}=0$ on an open portion of the boundary of $\partial W_1$.  This is a contradiction which proves (\ref{claimeqn}).

To complete the proof of the theorem, we note that for any sequence $\ve_j \rightarrow 0$, the harmonic functions $v^{\ve_j}_1$, $v^{\ve_j}_2$ and their derivatives are uniformly bounded in a neighborhood of the boundary $\partial \Omega$.  In particular, we can pass to a subsequence $\ve_{j_k}$ such that
\begin{equation} \label{useful}
-\int_{\partial \Omega} \frac{\partial v^{\ve_{j_k}}_1}{\partial \nu} \rightarrow a_1, \quad  -\int_{\partial \Omega} \frac{\partial v^{\ve_{j_k}}_2}{\partial \nu} \rightarrow a_2, \textrm{ as } k\rightarrow \infty,
\end{equation}
for bounded constants $a_1, a_2$ with $a_1, a_2>0$ (by Lemma \ref{proveass}).

Recalling (\ref{claimeqn}), we may assume without loss of generality that on the open portion $P$ of $\partial \Omega$ we have:
$$ - \frac{\partial v^{\ve}_1}{\partial \nu} + \frac{\partial v^{\ve}_2}{\partial \nu} \ge c>0.$$
If $a_2\ge a_1$, let $\varphi$ be a smooth nonnegative function supported on $P$ with $\varphi \ge 1$ on an open set $S\subset P$.  Then
$$-a_2 \varphi  \frac{\partial v^{\ve}_1}{\partial \nu} \ge -a_1 \varphi  \frac{\partial v^{\ve}_2}{\partial \nu} + ca_2, \quad \textrm{on } S,$$
and
$$-a_2 \varphi  \frac{\partial v^{\ve}_1}{\partial \nu} \ge -a_1 \varphi  \frac{\partial v^{\ve}_2}{\partial \nu}, \quad \textrm{ on } \partial \Omega,$$
and using (\ref{useful}) it follows that for $k$ sufficiently large, $Q_{\ve_{j_k}}(\varphi) \le -c'$ for a uniform constant $c'>0$.

The argument in the case when $a_2< a_1$ is similar except that we take $\varphi$ to have the opposite sign and obtain $Q_{\ve_{j_k}}(\varphi) \ge c'$.  

We have shown that for any sequence $\ve_j\rightarrow 0$ there is a subsequence $\ve_{j_k}$ such that $|Q_{\ve_{j_k}}(\varphi)| \ge c>0$.  Arguing by contradiction, this implies that $|Q_{\ve}(\varphi)| \ge c>0$ for all $\ve>0$ sufficiently small, after possibly shrinking $c>0$.
\end{proof}

We end with a remark on the argument for optimality of estimates in \cite[Section 3]{BLY}.  Bao-Li-Yin assume that $\Omega$ has a reflective symmetry and that the strictly convex set $D^0_2$ is a reflection of $D^0_1$, and, using a different argument, obtain examples for a large class of boundary data $\varphi$.  They also consider the case of $\Omega=\mathbb{R}^n$ (see \cite[Proposition 3.2]{BLY}).

\section{Examples} \label{sectionexamples}

We illustrate the above results in some special configurations in $\mathbb{R}^3$, using coordinates $x,y,z$. We  indicate briefly how Theorem \ref{main thm} can be applied in each case.

\subsection{A parabolic cylinder and a quartic cylinder}

Consider an example in $\mathbb{R}^3$ where $D_1$ and $D_2$ are
only close to each other near the origin,
and near there $\partial D_2$ is given locally as the graph
of the parabolic cylinder
$z = g := \varepsilon + x^2$
and $\partial D_1$ is the graph of the
quartic cylinder 
$z = f := - y^4$.
Then $g-f = \varepsilon +x^2 - y^4$,
so that $\alpha = (1,2)$ and $\gamma = \frac 12 + \frac 14 = \frac 34$.
Then in this configuration we find
\[
 \frac{|Q_{\ve}(\varphi)|}{C \ve^{3/4}} \le  \sup_{\tilde{\Omega}} | \nabla u | \le \frac C {\varepsilon^{3/4}} .
\]

\subsection{Encircled tori}

Consider two tori, with one tightly ringed around the other.
See figure \ref{fig: encircled tori}.


\begin{figure}[h]
\begin{tikzpicture}
[
  declare function=
  {
    a = 2;
    b = 1;
    c = 0.75; 
    sm(\t) = c/(sqrt(b^2*(cos(\t r))^2+a^2*(sin(\t r))^2));
    inellx(\t) = a*cos(\t r) + sm(\t)*(-b)*cos(\t r);
    inelly(\t) = b*sin(\t r) + sm(\t)*(-a)*sin(\t r);
    aa = 1.6;
    bb = 1;
    beta = 0.4;
    xd2(\p,\q) = aa - sin(\p r)*(bb + beta*cos(\q r));
    yd2(\p,\q) = cos(\p r)*(bb + beta*cos(\q r));
    zd2(\q) = beta*sin(\q r);
  },
  every plot/.style={domain=0:2*pi, variable=\t, smooth},
  scale=2,
]

\filldraw [fill=white] ellipse [x radius=a, y radius=b];
\def\sa{0.7}
\draw plot [domain=\sa:pi-\sa] ( {inellx(\t)}, {inelly(\t)} );

\begin{scope}

\fill [white, even odd rule]
  (aa,0) circle [radius=bb+beta]
  (aa,0) circle [radius=bb-beta];

\draw (aa,0) circle [radius=bb+beta];
\draw (aa,0) circle [radius=bb-beta];
\draw
  (aa, bb+beta) arc
  [
    x radius=beta/2, y radius=beta,
    start angle=90, end angle=270
  ]
;
\draw [dotted]
  (aa, bb+beta) arc
  [
    x radius=beta/2, y radius=beta,
    start angle=90,  end angle=-90
  ]
;
\draw
  (aa, -bb+beta) arc
  [
    x radius=beta/2, y radius=beta,
    start angle=90, end angle=270
  ]
;
\draw [dotted]
  (aa, -bb+beta) arc
  [
    x radius=beta/2, y radius=beta,
    start angle=90,  end angle=-90
  ]
;
\end{scope}

\def\sa{0.2}
\fill [white]
  (a,0)
  plot [domain=2*pi:pi] ({a*cos(\t r)}, {b*sin(\t r)})
  -- ( {inellx(pi+\sa)}, {inelly(pi+\sa)} )
  plot [domain=pi+\sa:2*pi-\sa] ( {inellx(\t)}, {inelly(\t)} )
  -- (a,0)
;
\draw plot [domain=pi+\sa:2*pi-\sa] ( {inellx(\t)}, {inelly(\t)} );
\draw 
  (-a,0) arc [x radius=a, y radius=b, start angle=180, end angle=360];

\draw [<->]
  (a,0) -- (aa+bb-beta,0)
  node [midway, below] {$\varepsilon$};

\end{tikzpicture}
\caption{Encircled tori.}
\label{fig: encircled tori}
\end{figure}
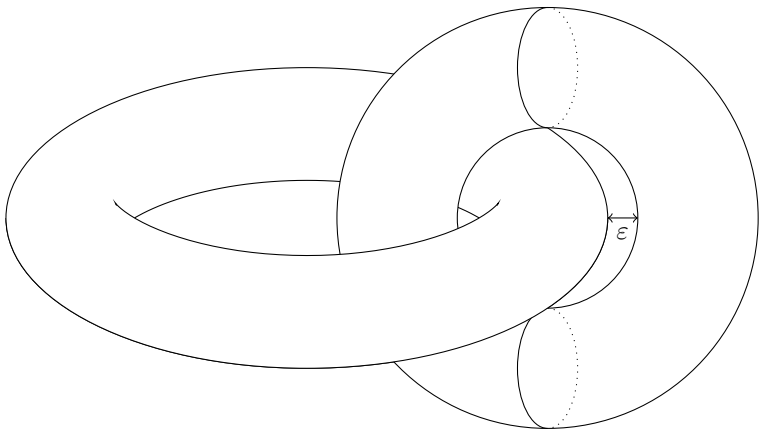


Specifically, let $D_2$ be the open solid torus in $\mathbb{R}^3$
bounded by
\[
  \bigl( \sqrt{ x^2 + y^2 } - a \bigr)^2 + z^2 = A^2,
\]
for constants $a>A>0$.  The boundary torus is the result of revolving the circle
$ (x-a)^2 + z^2 = A^2 $ in the $xz$-plane
about the $z$-axis.

Let $D_1$ be the region bounded by
\[
  \bigl( \sqrt{ (x-a)^2 + z^2 } - b \bigr)^2 + y^2 = B^2,
\]
for constants $b>B>0$, which is the circle $y^2 + (z-b)^2 = B^2$ in the
$x=a$ plane,
revolved about the line through $(a, 0,0)$ which is
parallel to the $y$-axis.

If we assume that $a \gg b$ and that
$b =A + \varepsilon + B$,
then $D_1$ is wrapped around $D_2$
with only $\varepsilon$ of distance between the two.

As $\ve \rightarrow 0$ the two boundaries intersect in a circle.
The narrow region, which is a neighborhood of this circle, can be covered with open boxes $U_i$ with $\alpha^i = (1, \infty)$ and $\gamma_i=1/2$.  We illustrate this in the case of a neighborhood of the point $(a-A, 0,0)$.

Near the point $(a-A, 0,0)$,
we can describe the boundaries of $D_1, D_2$
as functions of $y, z$.
Solving for $x$ in the equation for $D_2$ we find
that the boundary of $D_2$ is given locally as a graph
\[
  x = g(y,z) = a-A
  - \frac {y^2} {2(a-A)}
  + \frac {z^2} {2A}
  + \cdots
\]
$D_1$ is given locally as a graph
\[
  x = f(y,z) =
  a - b + B
  - \frac {y^2} {2B}
  + \frac {z^2} {2(b-B)}
  +  \cdots
\]
Compute
\[
  g-f = \varepsilon
  + \frac { a - A - B } { 2 B (a-A) } y^2
  + O(\ve|z|^2) + O(|y|^3),
\]
so that
\[
  \frac 1C (\varepsilon + y^2)
  < g-f
  < C (\varepsilon + y^2) .
\]
Hence the $\alpha^i$ and $\gamma_i$ corresponding to the box are given by  $(1,\infty)$ and $\frac12$ respectively.

The estimate we obtain for $\sup |\nabla u|$ is
\[
\frac{|Q_{\ve}(\varphi)|}{C\sqrt{\ve}} \le \sup_{\tilde{\Omega}} | \nabla u | \le \frac C {\sqrt\varepsilon} .
\]

\subsection{A torus and a sphere}

Consider the region in $\mathbb R^3$ inside the torus obtained
by revolving a planar disk of radius $A$ about an axis a
distance $a > A$ from the disk's center.
For example let $D_1$ be the region
\begin{equation}
  \notag 
  D_1 = \{ (x,y,z) \in \mathbb R^n \mid
  ( \sqrt{x^2+y^2} - a )^2 + z^2 < A^2 \} .
\end{equation}
Now suppose $D_2$ is the region inside a sphere of radius $R$
in $\mathbb R^3$, which is a distance $\ve$ from $D_1$.  The optimal estimate for $|\nabla u|$ will depend on where the sphere is centered.

Most configurations are already covered by the result of Bao-Li-Yin \cite{BLY0}.
For example, if the sphere were centered at
$(a+A+R+\varepsilon, 0, 0)$, 
then near the point $(a+A,0,0)$ 
the boundary of $D_1$ can be described to second order by
\begin{equation}
  \notag 
  x \approx a+A- C_1 y^2 - C_2 z^2
\end{equation}
and the boundary of $D_2$ is given, again to second order,
by
\begin{equation}
  \notag 
  x \approx a+A+ \ve+C_3 y^2 + C_4 z^2 
\end{equation}
where $C_1, \ldots, C_4$ are positive constants.  
Hence we find $\gamma = 1$ in this configuration, giving
$$
  \frac{ |Q_{\varepsilon}(\varphi)| }{ C \varepsilon |\log \ve| }
  \le \sup_{\tilde{\Omega}} | \nabla u |
  \le \frac { C } { \varepsilon |\log \ve| }
$$

On the other hand, a special configuration is when the sphere has radius $R = a-A-\varepsilon$, centered at the origin (namely, the sphere is  ``in the donut hole'').   See figure
\ref{sphere closely surrounded by a torus}.  Then we have an entire circle of close proximity
between the two regions.
In this case one can calculate $\gamma = \frac 12$ and thus
\begin{equation*}
  \notag 
  \frac{ |Q_{\varepsilon}(\varphi)| }{ C \sqrt\varepsilon }
  \le \sup_{\tilde{\Omega}}| \nabla u |
  \le \frac { C } { \sqrt\varepsilon },
\end{equation*}
as in the case of encircled tori above.


\begin{figure}[h]
\begin{tikzpicture}
[
  declare function=
  {
    a = 2;
    b = 1;
    c = 0.6; 
    sm(\t) = c/(sqrt(b^2*(cos(\t r))^2+a^2*(sin(\t r))^2));
    inellx(\t) = a*cos(\t r) + sm(\t)*(-b)*cos(\t r);
    inelly(\t) = b*sin(\t r) + sm(\t)*(-a)*sin(\t r);
  },
  every plot/.style={domain=0:2*pi, variable=\t, smooth},
  scale=2,
]

\filldraw [fill=white] ellipse [x radius=a, y radius=b];
\def\sa{0.2}
\draw plot [domain=\sa:pi-\sa] ( {inellx(\t)}, {inelly(\t)} );

\shade [ball color=gray!5] circle [radius=1.3];

\def\sa{0.05}
\fill [white]
  (a,0)
  plot [domain=2*pi:pi] ({a*cos(\t r)}, {b*sin(\t r)})
  -- ( {inellx(pi+\sa)}, {inelly(pi+\sa)} )
  plot [domain=pi+\sa:2*pi-\sa] ( {inellx(\t)}, {inelly(\t)} )
  -- (a,0)
;
\draw plot [domain=pi+\sa:2*pi-\sa] ( {inellx(\t)}, {inelly(\t)} );
\draw 
  (-a,0) arc [x radius=a, y radius=b, start angle=180, end angle=360];

\def\xx{0.27}
\draw plot [domain=0:180]
 ( 0, {\xx*cos(\t)}, {1.65+\xx*sin(\t)} );
\draw [dotted] plot [domain=180:360]
 ( 0, {\xx*cos(\t)}, {1.65+\xx*sin(\t)} );

\end{tikzpicture}
\caption{Sphere closely surrounded by a torus.}
\label{sphere closely surrounded by a torus}
\end{figure}
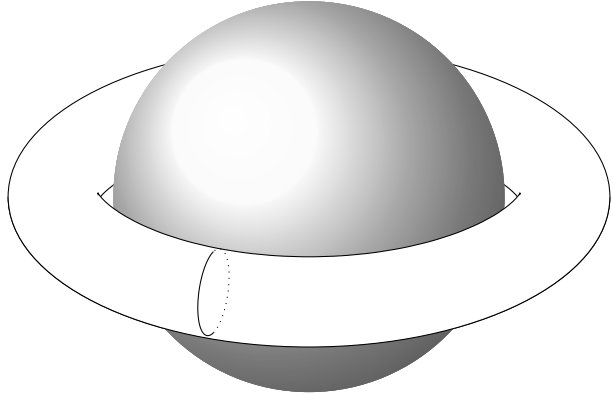


\end{document}